\documentclass[a4paper,12pt]{amsart}

\usepackage{amsmath, amssymb}
\usepackage{tikz}

\allowdisplaybreaks

\newtheorem{thm}{Theorem}[section]
\newtheorem{cor}[thm]{Corollary}
\newtheorem{lem}[thm]{Lemma}
\newtheorem{prop}[thm]{Proposition}

\newtheorem*{lem*}{Lemma}

\theoremstyle{definition}

\theoremstyle{remark}
\newtheorem{rmk}[thm]{Remark}

\newtheorem{example}[thm]{Example}

\numberwithin{equation}{section}
\numberwithin{figure}{section}

\newcommand{\NN}{\mathbb{N}}

\newcommand{\RR}{\mathbb{R}}

\newcommand{\TC}{\mathcal{T}C}
\newcommand{\Aut}{\operatorname{Aut}}

\date{\today}

\begin{document}

\author[Astrid an Huef]{Astrid an Huef}
\author[Iain Raeburn]{Iain Raeburn}
\email{astrid.anhuef@vuw.ac.nz, iain.raeburn@vuw.ac.nz}
\address{School of Mathematics and Statistics, Victoria University of Wellington, P.O. Box 600, Wellington 6140, New Zealand.}

\title[KMS states on higher-rank graph algebras]{Equilibrium states on the Toeplitz algebras\\ of small higher-rank graphs}

\begin{abstract}{We consider a family of operator-algebraic dynamical systems involving the Toeplitz algebras of higher-rank graphs. We explicitly compute the KMS states (equilibrium states) of these systems built from small graphs with up to four connected components.}
\end{abstract}

\maketitle

Over the past decade there has  been a great deal of interest in KMS states (or equilibrium states) on $C^*$-algebras of directed graphs \cite{KW, aHLRS1, Chl, CL, M, T} and their higher-rank analogues \cite{aHLRS2, aHLRS3, LLNSW, Yang3, FGKP, FGJKP, C}. At first this work focused on strongly connected graphs, possibly because these  had provided many examples of interesting simple $C^*$-algebras \cite{Sz, PRRS}. A uniform feature of all this work is that the Toeplitz algebras of graphs have much more interesting KMS structure than their Cuntz--Krieger quotients. 

More recently, we have been looking at graphs with more than one strongly connected component \cite{aHLRS4, aHKR2}. This throws up new problems: removing components can create sources, and, as Kajiwara and Watatani demonstrated in \cite[Theorem~4.4]{KW}, sources give rise to extra KMS states (see also \cite[Corollary~6.1]{aHLRS1} for a version phrased with our graph-algebra conventions). For higher-rank graphs, the situation is even more complicated. We saw in \cite[\S8]{aHKR2} that even if the original graph has no sources, removing a component can give sources of several different kinds. So when we tried to test the general results from \cite{aHKR} and \cite[\S3--6]{FaHR} by organising them into a coherent program for calculating KMS states, we ruled out the possiblity that there could be non-trivial bridges between different components \cite[\S7--8]{FaHR}. Christensen \cite{C2} has recently shown that there is in principle no need to do this: our problems arose because we want to find explicit formulas for the KMS states.

So here we study graphs like the ones in \cite[\S8]{aHKR2} which led us to rule out non-trivial bridges. We find that, while there are indeed new difficulties at almost every turn, the  general results in \cite[\S3--6]{FaHR} do in fact suffice to determine all the KMS states.

The first problem we run into concerns the  different kinds of sources: some are \emph{absolute sources}, which receive no edges, and others may receive edges of some colours but not others. The only Cuntz--Krieger relations available for higher-rank graphs with sources are those of \cite{RSY2}; since these are complicated, we have to develop some techniques for finding efficient sets of relations, and these techniques may be useful elsewhere. We discuss this in \S\ref{exhsets}, and introduce the two main families of $2$-graphs that we will analyse. The first family contains graphs with two vertices $\{u,v\}$;  the second, graphs with three vertices $\{u,v,w\}$, each containing one of the first family as a subgraph. We normalise the dynamics to ensure  that the critical inverse temperature is always~$1$. 

We observed in \cite[\S8]{aHKR2} that the results of \cite{aHLRS2} for non-critical $\beta$ do apply in the presence of sources even though there was officially a blanket assumption of ``no sources'' in \cite{aHLRS2}. In all our classification results, the KMS states at non-critical inverse temperatures are parametrised by a simplex of ``subinvariant vectors''.  Identifying this simplex requires computing the numbers of paths with range a given vertex, and the presence of sources throws up new problems, which we deal with in \S\ref{secex2.2} and \S\ref{sectrickyex}. The results in \S\ref{sectrickyex} for the graphs with three vertices depend on the analogous results for the graphs with two vertices in the previous section.

In \S\ref{seccritinvt} we apply the results of \cite[\S4 and \S5]{FaHR} to find the KMS states at the critical inverse temperature for the  higher-rank graphs we studied in the preceding sections. This involves proving, first for the example with two vertices $\{u,v\}$, and then for the three-vertex example, that a KMS$_1$ state cannot see any of the vertices except $u$. The proofs rely on our understanding of the Cuntz--Krieger relations for graphs with sources.

The examples we have studied arose in \cite[\S8]{aHKR2} as subgraphs of a graph with $4$ vertices which does not itself have sources. We finish by calculating the KMS$_1$ states of this example. We get one by lifting the unique KMS$_1$ state of $\TC^*(u\Lambda u)$, and we find a second by stepping carefully through the construction of  \cite[\S4]{FaHR}. In the final section, we investigate what happens below the critical inverse temperature.

In conclusion: while the existence of sources in a graph or in subgraphs certainly complicates the situation, it doesn't necessarily make the situation intractable.

\section{Exhaustive sets}\label{exhsets}

Throughout, we consider a finite $k$-graph $\Lambda$, and typically $k=2$. A vertex $v\in \Lambda^0$ is a \emph{source} if there exists $i\in \{1,\dots,k\}$ such that $v\Lambda^{e_i}$ is empty. (We believe this is standard: it is the negation of ``$\Lambda$ has no sources'' in the sense of the original paper \cite{KP}.) Our main examples are $2$-graphs in which a vertex can receive red edges but not blue edges or vice-versa --- for example, the vertex $v$ in the graph in \cite[Figure~2]{aHKR2} (which is also Example~\ref{trickyex} below). We call a vertex such that $v\Lambda^{e_i}$ is empty for all $i$ an \emph{absolute source}. For example, the vertex $w$ in Example~\ref{trickyex} is an absolute source.

For such graphs the only Cuntz--Krieger relations available are those of \cite{RSY2}. As there, if $\mu, \nu$ is a pair of of paths in $\Lambda$ with $r(\mu)=r(\nu)$, we set
\[\Lambda^{\min}(\mu,\nu)=\big\{(\alpha,\beta)\in \Lambda\times \Lambda: \mu\alpha=\nu\beta\text{ and } d(\mu\alpha)=d(\mu)\vee d(\nu)\big\}
\]
for the set of \emph{minimal common extensions}. For a vertex $u\in \Lambda^0$, a finite subset $E$ of $u\Lambda^1:=\bigcup_{i=1}^ku\Lambda^{e_i}$ is \emph{exhaustive} if for every $\mu\in u\Lambda$ there exists $e\in E$ such that $\Lambda^{\min}(\mu,e)$ is nonempty. 

We then use the Cuntz--Krieger relations of \cite{RSY2}, and in particular the presentation of these relations which uses only edges,   as discussed in \cite[Appendix~C]{RSY2}. As there, we write $\{t_\mu:\mu\in \Lambda\}$ for the universal Toeplitz--Cuntz--Krieger family which generates the Toeplitz algebra $\TC^*(\Lambda)$. The Cuntz--Krieger relations then include the relations (T1), (T2), (T3) and (T5) for Toeplitz--Cuntz--Krieger families, as in \cite{aHLRS2} and \cite{aHKR2}, and for every $u$ that is not a source the extra relations 
\begin{itemize}
\item[(CK)] $\prod_{e\in E}(t_u-t_et_e^*)=0$\quad\text{for all $u\in \Lambda^0$ and finite exhaustive sets $E\subset u\Lambda^1$}.
\end{itemize}
Since adding extra edges to an exhaustive set gives another exhaustive set, it is convenient to find the smallest possible exhaustive sets. Then the Cuntz--Krieger relations for these smallest sets are the sharpest.

\begin{lem}\label{idFE}
Suppose that $\Lambda$ is a finite $k$-graph and $E\subset u\Lambda^1$ is a finite exhaustive set. Consider $i\in\{1,\dots,k\}$ and $e\in u\Lambda^{e_i}$. If there is a path $e\mu\in u\Lambda^{\NN e_i}$ such that $s(\mu)$ is an absolute source, then $e\in E$.
\end{lem}

\begin{proof}
Since $E$ is exhaustive, there exists $f\in E$ such that $\Lambda^{\min}(e\mu,f)\not=\emptyset$. Then there exist paths $\alpha,\beta$ such that $e\mu\alpha=f\beta$. Since $s(\mu)\Lambda=\{s(\mu)\}$, we have $\alpha=s(\mu)$. Thus $d(f\beta)=d(e\mu\alpha)=d(e\mu)\in \NN{e_i}$;  since $f\in\Lambda^1$ and
\[
0\leq d(f)\leq d(f)+d(\beta)=d(f\beta) \in \NN{e_i},
\]
we deduce that $d(f)=e_i$. Now uniqueness of factorisations and $e\mu=f\beta$ imply that $f=(e\mu)(0,e_1)=e$. Thus $e\in E$.
\end{proof}

\begin{example}\label{ex2graph1source}
We consider a $2$-graph $\Lambda$ with the following skeleton: 
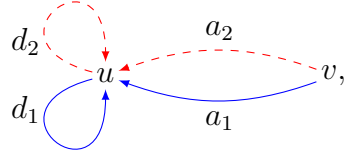
\[
\begin{tikzpicture}[scale=1.5]
 \node[inner sep=0.5pt, circle] (u) at (0,0) {$u$};
    \node[inner sep=0.5pt, circle] (v) at (2,0) {$v,$};
\draw[-latex, blue] (u) edge [out=195, in=270, loop, min distance=30, looseness=2.5] (u);
\draw[-latex, red, dashed] (u) edge [out=165, in=90, loop, min distance=30, looseness=2.5] (u);
\draw[-latex, blue] (v) edge [out=200, in=340] (u);
\draw[-latex, red, dashed] (v) edge [out=160, in=20] (u) ;
\node at (-.7, 0.3) {\color{black} $d_2$};
\node at (-.7,-0.3) {\color{black} $d_1$};
\node at (1, 0.4) {\color{black} $a_2$};
\node at (1,-0.4) {\color{black} $a_1$};
\end{tikzpicture}
\]
innwhcih the label $a_1$, for example, means that there are $a_1$ blue edges from $v$ to $u$. Since each path in $u\Lambda^{e_1+e_2}v$ has unique blue-red and red-blue factorisations, the numbers $a_i, d_i\in \NN\backslash\{0\}$ satisfy $d_1a_2=d_2a_1$. Since $v$ is an absolute source, Lemma~\ref{idFE} implies that every finite exhaustive set in $u\Lambda^1$ must contain every edge, and hence contains $u\Lambda^1$; since every finite exhaustive set is by definition a subset of $u\Lambda^1$, it is therefore the only finite exhaustive set. Thus the only Cuntz--Krieger relation at $u$ is
\[
\prod_{e\in u\Lambda^1}(t_u-t_et_e^*)=0.
\]
There is no Cuntz--Krieger relation at $v$.
 \end{example}
 
We now aim to make the Cuntz--Krieger relation (CK) look a little more like the familiar ones involving sums of range projections. 

\begin{prop}\label{CKforE}
Suppose that $\Lambda$ is a finite  $k$-graph, $u$ is a vertex and $E\subset u\Lambda^1$ is a finite exhaustive set. For each nonempty subset $J$ of $\{1,\dots,k\}$, define $e_J\in \NN^k$ by $e\!_J=\sum_{i\in J}e_i$. Then the Cuntz--Krieger relation \textnormal{(CK)} associated to $E$ is equivalent to
\[
t_u+\sum_{\emptyset\not=J\subset \{1,\dots,k\}}(-1)^{|J|}\sum_{\{\mu\in u\Lambda^{e\!_J}:\mu(0,e_i)\in E\text{ for $i\in J$}\}} t_\mu t_\mu^* =0.
\]
\end{prop}

From the middle of \cite[page~120]{aHKR2} we have
\begin{equation}\label{expCK}
\prod_{e\in E}(t_u-t_et_e^*)=t_u+\sum_{\emptyset\not=J\subset \{1,\cdots,k\}} (-1)^{|J|}\prod_{i\in J}\Big(\sum_{e\in E\cap u\Lambda^{e_i}}t_et_e^*\Big).
\end{equation}
We want to expand the product, and we describe the result in a lemma. Proposition~\ref{CKforE} then follows from \eqref{expCK} and the lemma.
 
 \begin{lem}\label{expprod}
Suppose that $E$ is a finite exhaustive subset of $u\Lambda^1$ and $J$ is a subset of $\{1,\cdots,k\}$. Write $e_J=\sum_{i\in J}e_i$. Then
\begin{equation}\label{CKEassum}
\prod_{i\in J} \Big(\sum_{f\in E\cap u\Lambda^{e_i}}t_ft_f^*\Big)=\sum_{\{\mu\in u\Lambda^{e\!_J}:\mu(0,e_i)\in E\text{ for $i\in J$}\}}t_\mu t_\mu^*.
\end{equation}
\end{lem}

\begin{proof}
For $|J|=1$, says $J=\{i\}$, we have $e_J=e_i$. Thus
\[
\{\mu\in u\Lambda^{e\!_J}:\mu(0,e_i)\in E\text{ for $i\in J$}\}=\{\mu\in u\Lambda^{e_i}:\mu=\mu(0,e_i)\in E\}= E\cap uE^{e_1}.
\]
Now suppose the formula holds for $|J|=n$ and that $K:=J\cup\{j\}$ for some $j\in\{1,\dots,k\}\setminus J$. Then the inductive hypothesis gives
\begin{align*}
\prod_{i\in K} \Big(\sum_{f\in E\cap u\Lambda^{e_i}}t_ft_f^*\Big)
&=\bigg(\prod_{i\in J} \Big(\sum_{f\in E\cap u\Lambda^{e_i}}t_ft_f^*\Big)\bigg)\Big(\sum_{e\in E\cap u\Lambda^{e_j}}t_et_e^*\Big)\\
&=\Big(\sum_{\{\mu\in u\Lambda^{e\!_J}:\mu(0,e_i)\in E\text{ for $i\in J$}\}}t_\mu t_\mu^*\Big)\Big(\sum_{e\in E\cap u\Lambda^{e_j}}t_et_e^*\Big).
\end{align*}
Now for each pair of summands $t_\mu t_\mu^*$ and $t_et_e^*$ the relation (T5) gives
\[
(t_\mu t_\mu^*)(t_et_e^*)=\sum_{(g,\nu)\in\Lambda^{\min}(\mu,e)}t_\mu(t_g t_\nu^*)t_e^*
=\sum_{(g,\nu)\in\Lambda^{\min}(\mu,e)}t_{\mu g} t_{\nu e}^*.
\]
By definition of $\Lambda^{\min}$ we have $g\in \Lambda^{e_j}$ and $\mu g=e\nu$, so $d(\mu g)=e_K$. Then we have
\[
(\mu g)(0,e_j)=(e\nu)(0,e_j)=e\in E\quad\text{and}\quad (\mu g)(0,e_i)=\mu(0,e_i)\in E \text{ for $i\in J$.}
\]
So the paths which arise as $\mu g$ are precisely those in the set 
\[
\{\lambda\in u\Lambda^{e_K}:\lambda(0,e_i)\in E\text{ for $i\in J\cup \{j\}=K$}\}.\qedhere
\]
\end{proof}

\begin{rmk}
It is possible that the index set on the right-hand side of \eqref{CKEassum} is empty, in which case we are asserting that the product on the left is $0$.
\end{rmk}

In a finite $k$-graph, the set $u\Lambda^1$ of all edges with range $u$ is always a finite exhaustive subset of $\Lambda$. Then Proposition~\ref{CKforE} applies with $E=u\Lambda^1$. For this choice of $E$, the condition $\mu(0,e_i)\in E$ is trivially satisfied, and hence we get the following simpler-looking relation.

\begin{cor}\label{FEex2.2}
Suppose that $\Lambda$ is a finite $k$-graph. Then for every $u\in \Lambda^0$ we have
\[
\prod_{f\in u\Lambda^1}(t_u-t_ft_f^*)= t_u+\sum_{J\subset\{1,\dots,k\}} (-1)^{|J|}\Big(\sum_{\mu\in \Lambda^{e\!_J}}t_\mu t_\mu^*\Big).
\]
\end{cor}

\begin{example}\label{ex2graph1sourceFE}
We return to a $2$-graph $\Lambda$ with the skeleton desccribed in Example~\ref{ex2graph1source}, and its only finite exhaustive set  $u\Lambda^1$. Then \eqref{expCK} and Lemma~\ref{expprod} imply that 
\[
\prod_{e\in u\Lambda^1}(t_u-t_et_e^*)=t_u+\sum_{\emptyset\not=J\subset \{1,2\}}(-1)^{|J|}\Big(\sum_{\{\mu\in \Lambda^{e\!_J}:\mu(0,e_i)\in u\Lambda^1\}}t_\mu t_\mu^*\Big).
\]
The nonempty subsets of $\{1,2\}$ are $\{1\}$, $\{2\}$ and $\{1,2\}$. For $J=\{1\}$ the requirement $\mu(0,e_1)\in u\Lambda^1$ just says that $\mu$ is a blue edge ($d(\mu)=e_1$), and 
\[
\sum_{\mu\in \Lambda^{e\!_J}}t_\mu t_\mu^*=\sum_{e\in u\Lambda^{e_1}}t_et_e^*.
\]
A similar thing happens for $J=\{2\}$. For $J=\{1,2\}$, the condition on $\mu(0, e_i)$ is still trivially satisfied by all $\mu\in u\Lambda^{e_1+e_2}$. Hence the Cuntz--Krieger relation becomes
\[
t_u-\sum_{e\in u\Lambda^1}t_et_e^*-\sum_{e\in u\Lambda^2}t_et_e^*+\sum_{\mu\in u\Lambda^{e_1+e_2}}t_\mu t_\mu^*=0,
\]
or equivalently
\[
t_u=\sum_{e\in u\Lambda^1}t_et_e^*+\sum_{e\in u\Lambda^2}t_et_e^*-\sum_{\mu\in u\Lambda^{e_1+e_2}}t_\mu t_\mu^*,
\]
which does indeed look more like a Cuntz--Krieger relation. 
\end{example}

Lemma~\ref{idFE} establishes a lower bound for the  finite exhaustive sets. In Example~\ref{ex2graph1source}, this lower bound was all of $u\Lambda^1$, and hence this had to be the only finite exhaustive set. However, this was a bit lucky, as the next example shows.

\begin{example}\label{trickyex}
We consider a $2$-graph $\Lambda$ with skeleton
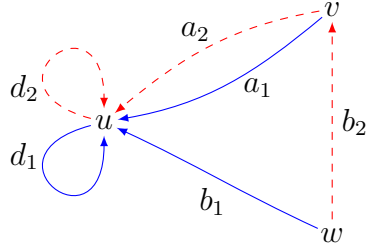
\[
\begin{tikzpicture}[scale=1.5]
 \node[inner sep=0.5pt, circle] (u) at (0,0) {$u$};
    \node[inner sep=0.5pt, circle] (v) at (2,1) {$v$};
    \node[inner sep=0.5pt, circle] (w) at (2,-1) {$w$};
\draw[-latex, blue] (u) edge [out=195, in=270, loop, min distance=30, looseness=2.5] (u);
\draw[-latex, red, dashed] (u) edge [out=165, in=90, loop, min distance=30, looseness=2.5] (u);
\draw[-latex, blue] (v) edge [out=220, in=10] (u);
\draw[-latex, red, dashed] (v) edge [out=190, in=40] (u) ;
\draw[-latex, blue] (w) edge [out=155, in=335] (u) ;
\draw[-latex, red, dashed] (w) edge [out=90, in=270] (v);
\node at (-.7, 0.3) {\color{black} $d_2$};
\node at (-.7,-0.3) {\color{black} $d_1$};
\node at (.8, 0.8) {\color{black} $a_2$};
\node at (1.35,0.3) {\color{black} $a_1$};
\node at (2.2, 0) {\color{black} $b_2$};
\node at (.95,-0.7) {\color{black} $b_1$};
\end{tikzpicture}
\]
in which $d_1a_2=d_2a_1$ and $a_1b_2=d_2b_1$. Note that $w$ is an absolute source. Our interest in these graphs arises from \cite[\S8]{aHKR2}, where we saw that the Toeplitz algebras of such graphs can arise as quotients of the Toeplitz algebras of graphs with no sources.

The only finite exhaustive subset of $v\Lambda^1=v\Lambda^{e_2}$ is the whole set, and this yields a Cuntz--Krieger relation 
\begin{equation}\label{CKv}
\prod_{f\in v\Lambda^{e_2}}(t_v-t_ft_f^*)=0\Longleftrightarrow t_v=\sum_{f\in v\Lambda^{e_2}}t_ft_f^*.
\end{equation}
For the vertex $u$ the situation is more complicated.
\end{example}

\begin{prop}\label{idminFE} 
Suppose that $\Lambda$ is a $2$-graph with the skeleton described in Example~\ref{trickyex}. Then 
\[
E:= (u\Lambda^1u)\cup(u\Lambda^{e_2}v)\cup(u\Lambda^{e_1}w)
\]
is exhaustive, and every other finite exhaustive subset of $u\Lambda^1$ contains $E$. 
\end{prop}

\begin{proof}
Since $w$ is an absolute source, Lemma~\ref{idFE} implies  that every finite exhaustive subset of $u\Lambda^1$ contains $u\Lambda^1u$, $u\Lambda^{e_2}v$ and $u\Lambda^1w=u\Lambda^{e_1}w$, and hence also the union $E$. So it suffices for us to prove that $E$ is exhaustive.

To see this, we take $\lambda\in u\Lambda$ and look for $e\in E$ such that $\Lambda^{\min}(\lambda,e)\not=\emptyset$. Unfortunately, this seems to require a case-by-case argument. We begin by eliminating some easy cases.
\begin{itemize}
\item If $\lambda=u$, we take $e\in u\Lambda^1u$; then $e\in \Lambda^{\min}(\lambda,e)$, and we are done. So we suppose that $d(\lambda)\not=0$.
\smallskip

\item If $\lambda\in u\Lambda u\setminus\{u\}$, we choose $i$ such that $e_i\leq d(\lambda)$. Then $\lambda(0,e_i)\in u\Lambda ^{e_i}u\subset E$ and $\Lambda^{\min}(\lambda,\lambda(0,e_i))=\{\lambda\}$ is nonempty.
\smallskip

\item We now suppose that $\lambda\in u\Lambda v\cup u\Lambda w$. If $e_2\leq d(\lambda)$, then $\lambda(0,e_2)\in E$ and $\Lambda^{\min}(\lambda, \lambda(0,e_2))\not=\emptyset$.  Otherwise $d(\lambda)\in\NN{e_1}$. 
\smallskip

\item If $d(\lambda)\geq 2e_1$, then $\lambda(0,e_1)\in u\Lambda^{e_1} u$ belongs to $E$, and $\Lambda^{\min}(\lambda, \lambda(0,e_1))\not=\emptyset$.
\smallskip

\item If $d(\lambda)=e_1$ and $s(\lambda)=w$, then $\lambda\in u\Lambda^{e_1} w$ belongs to $E$, and we take $e=\lambda$.
\end{itemize}
We are left to deal with paths $\lambda\in u\Lambda^{e_1}v$.  Choose $f\in v\Lambda^{e_2}w$, and consider $\lambda f$. Since $d(\lambda f)=d(\lambda)+d(f)=e_1+e_2$, $\lambda f$ has a red-blue factorisation 
\[
\lambda f=(\lambda f)(0,e_2)(\lambda f)(e_2,e_1+e_2). 
\]
But now $(\lambda f)(0,e_2)\in u\Lambda^{e_2}u\subset E$, and we have 
\[
(f,(\lambda f)(e_2,e_1+e_2)\big)\in\Lambda^{\min}\big(\lambda,(\lambda f)(0,e_2)\big). 
\]
Thus in all cases $\lambda$ has a common extension with some edge in $E$, and $E$ is exhaustive.
\end{proof}

So for the graphs $\Lambda$ with skeleton described in Example~\ref{trickyex}, there is a single Cuntz--Krieger relation at the vertex $u$, namely $\prod_{e\in E}(t_u-t_et_e^*)=0$. Now we rewrite this relation as a more familiar-looking sum. 

\begin{lem}\label{CKatu}
Suppose that $\Lambda$ is a $2$-graph with skeleton described in Example~\ref{trickyex}, and $E$ is the finite exhaustive set of Lemma~\ref{idminFE}. Then we have
\begin{equation}\label{expandprod_E}
\prod_{e\in E}(t_u-t_et_e^*)=t_u-\sum_{e\in u\Lambda^{e_1}\{u,w\}}t_et_e^*-\sum_{f\in u\Lambda^{e_2}\{u,v\}}t_ft_f^*+\sum_{\mu\in u\Lambda^{e_1+e_2}\{u,v\}}t_\mu t_\mu^.
\end{equation}
\end{lem}

\begin{proof}
From \eqref{expCK} and Lemma~\ref{expprod} we deduce that
\begin{align*}
\prod_{e\in E}(t_u&-t_et_e^*)=t_u-\sum_{e\in u\Lambda^{e_1}\cap E}t_et_e^*-\sum_{f\in u\Lambda^{e_2}\cap E}t_ft_f^*+\sum_{\{\mu\in \Lambda^{e_1+e_2}:\mu(0,e_i)\in E\text{ for $i=1,2$}\}}t_\mu t_\mu^*\\
&=t_u-\sum_{e\in u\Lambda^{e_1}\{u,w\}}t_et_e^*-\sum_{f\in u\Lambda^{e_2}\{u,v\}}t_ft_f^*+\sum_{\{\mu\in \Lambda^{e_1+e_2}:\mu(0,e_i)\in E\text{ for $i=1,2$}\}}t_\mu t_\mu^*.
\end{align*}
To understand the last term, we claim that $\mu\in u\Lambda^{e_1+e_2}$ has $\mu(0,e_1)\in E$ and $\mu(0,e_2)\in E$ if and only if $s(\mu)=u$ or $s(\mu)=v$. The point is that if  $s(\mu)=u$ or $s(\mu)=v$ then $s(\mu(0,e_i))=u$ for $i=1,2$, and $u\Lambda^1u\subset E$. The alternative is that $s(\mu)=w$, and then $\mu(0,e_1)$ belongs to $u\Lambda^{e_1}v$, which is not in $E$. Thus 
\[
\{\mu\in \Lambda^{e_1+e_2}:\mu(0,e_i)\in E\text{ for $i=1,2$}\}=u\Lambda^{e_1+e_2}\{u,v\},
\]
and this completes the proof.
\end{proof}

\begin{cor}
Suppose that $\Lambda$ is a $2$-graph with skeleton described in Example~\ref{trickyex}. Then the Cuntz--Krieger algebra is the quotent of $\TC^*(\Lambda)$ by the Cuntz--Krieger relations \eqref{CKv} and 
\[
t_u=\sum_{e\in u\Lambda^{e_1}\{u,w\}}t_et_e^*+\sum_{f\in u\Lambda^{e_2}\{u,v\}}t_ft_f^*-\sum_{\mu\in u\Lambda^{e_1+e_2}\{u,v\}}t_\mu t_\mu^*.
\]
\end{cor}

\section{KMS states for the graphs of Example~\ref{ex2graph1source}}\label{secex2.2}

We wish to compute the KMS$_\beta$ states for a 2-graph $\Lambda$ with skeleton described in Example~\ref{ex2graph1source}. Such graphs have one absolute source $v$. We list the vertex set as $\{u,v\}$, and write $A_i$ for the vertex matrices, so that
\[
A_i=\begin{pmatrix}d_i&a_i\\0&0\end{pmatrix}\quad\text{for $i=1,2$.}
\]
We then fix $r\in (0,\infty)^2$, and consider the associated  dynamics $\alpha^r:\RR\to \Aut \TC^*(\Lambda)$ such that
\[
\alpha_t(t_\mu t_\nu^*)=e^{itr\cdot (d(\mu)-d(\nu))}t_\mu t_\nu^*.
\]
We then consider $\beta\in (0,\infty)$ such that 
\begin{equation}\label{hypbigbeta}
\beta r_i>\ln\rho(A_i)\quad\text{for $i=1$ and $i=2$.}
\end{equation} 
As observed at the start of \cite[\S8]{aHKR2}, even though $\Lambda$ has a source, we can still apply Theorem~6.1 of \cite{aHLRS2} to find the KMS$_\beta$ states. 

First we need to compute the vector 
$y=(y_u,y_v)\in [0,\infty)^{\Lambda^0}$ appearing in that theorem. We find:

\begin{lem}\label{lemcompy}
We have
\begin{align}
y_u&=\sum_{\mu\in \Lambda u}e^{-\beta r\cdot d(\mu)}=(1-d_1e^{-\beta r_1})^{-1}(1-d_2e^{-\beta r_2})^{-1},\quad \text{and}\label{ysubu}\\
y_v&=1+a_1e^{-\beta r_1}(1-d_1e^{-\beta r_1})^{-1}(1-d_2e^{-\beta r_2})^{-1}+a_2e^{-\beta r_2}(1-d_2e^{-\beta r_2})^{-1}.\label{ysubv}
\end{align}
\end{lem}

\begin{proof}
We first evaluate
\[
y_u:=\sum_{\mu\in \Lambda u}e^{-\beta r\cdot d(\mu)}=\sum_{n\in \NN^2}\sum_{\mu\in \Lambda^n u}e^{-\beta r\cdot n}.
\]
Each path of degree $n$ is uniquely determined by (say) its blue-red factorisation. Then  we have $d_1^{n_1}$ choices for the blue path and $d_2^{n_2}$ choices for the red path. Thus
\begin{align*}
y_u&=
\sum_{n\in \NN^2} d_1^{n_1}d_2^{n_2}e^{-\beta(n_1r_1+n_2r_2)}=\sum_{n\in \NN^2}(d_1e^{-\beta r_1})^{n_1}(d_2e^{-\beta r_2})^{n_2}\\
&=\Big(\sum_{n_1=0}^\infty (d_1e^{-\beta r_1})^{n_1}\Big)\Big(\sum_{n_2=0}^\infty (d_2e^{-\beta r_2})^{n_2}\Big), 
\end{align*}
and summing the geometric series gives \eqref{ysubu}.

To compute $y_v$, we need to list the distinct paths $\mu$ in $\Lambda v$. First, if $d(\mu)_1>0$, then $\mu$ has a factorisation $\mu=\nu f$ with $d(f)=e_1$. Note that $s(f)=s(\mu)=v$, and hence $s(\nu)=r(f)=u$, so $\nu\in \Lambda u$. Otherwise we have $d(\mu)\in \NN{e_2}$, and $\Lambda v$ is the disjoint union of the singleton $\{v\}$, $\bigcup_{e\in \Lambda^{e_1}v}(\Lambda u)e$,  and $\bigcup_{l=0}^\infty\Lambda^{(l+1)e_2}v$. Counting the three sets gives 
\[
y_v=1+a_1e^{-\beta r_1} y_u+\sum_{l=0}^\infty a_2d_2^le^{-\beta(l+1)r_2}=1+a_1e^{-\beta r_1}y_u+a_2e^{-\beta r_2}(1-d_2e^{-\beta r_2})^{-1},
\]
and hence we have \eqref{ysubv}. \end{proof}

\begin{rmk}\label{rmkaltformyv}
We made a choice when we computed $y_v$: we considered the complementary cases $d(\mu)_1>0$ and $d(\mu)_1=0$. We could equally well have chosen to use the cases $d(\mu)_2>0$ and $d(\mu)_2=0$, and we would have found
\begin{equation}\label{altformyv}
y_v=1+a_2e^{-\beta r_2}(1-d_1e^{-\beta r_1})^{-1}(1-d_2e^{-\beta r_2})^{-1}+a_1e^{-\beta r_1}(1-d_1e^{-\beta r_1})^{-1},
\end{equation}
which looks different. To see that they are in fact equal, we look at the difference. To avoid messy formulas, we write $\Delta:=(1-d_2e^{-\beta r_2})(1-d_2e^{-\beta r_2})$, and observe that, for example, $(1-d_1e^{-\beta r_1})^{-1}=(1-d_2e^{-\beta r_2})\Delta^{-1}$. Then the difference $\eqref{ysubv} -\eqref{altformyv}$ is
\begin{align*}
a_1e^{-\beta r_1}\Delta^{-1}&+a_2e^{-\beta r_2}(1-d_1e^{-\beta r_1})\Delta^{-1}-
a_2e^{-\beta r_2}\Delta^{-1}-a_1e^{-\beta r_1}(1-d_2e^{-\beta r_2})\Delta^{-1}.
\end{align*}
When we expand the brackets we find that the terms $a_1e^{-\beta r_1}\Delta^{-1}$ and $a_2e^{-\beta r_2}\Delta^{-1}$ cancel out, leaving 
\[
-a_2e^{-\beta r_2}d_1e^{-\beta r_1}\Delta^{-1}+a_1e^{-\beta r_1}d_2e^{-\beta r_2}\Delta^{-1}=(-a_2d_1+a_1d_2)e^{-\beta(r_1+r_2)}\Delta^{-1},
\]
which vanishes because the factorisation property forces $a_1d_2=a_2d_1$.
\end{rmk}

We recall that we are considering $\beta$ satisfying \eqref{hypbigbeta}.
The first step in the procedure of \cite[\S8]{FaHR} for such $\beta$ is to apply \cite[Theorem~6.1]{aHLRS2}. Then the KMS states of $(\TC^*(\Lambda),\alpha^r)$ have the form $\phi_\epsilon$ for $\epsilon\in [0,\infty)^{\{u,v\}}$ satisfying $\epsilon\cdot y=1$. This is a $1$-dimensional simplex with extreme points $(y_u^{-1},0)$ and $(0,y_v^{-1})$. The values of the state $\phi_\epsilon$ on the vertex projections $t_u$ and $t_v$ are the coordinates of the vector
\[
m=\Big(\prod_{i=1}^2(1-e^{-\beta r_i}A_i)^{-1}\Big)\epsilon.
\]
To find $m$, we compute
\[
\prod_{i=1}^2(1-e^{-\beta r_i}A_i)^{-1}=\Big(\prod_{i=1}^2(1-d_ie^{-\beta r_i})^{-1}\Big)\begin{pmatrix}
1&a_2e^{-\beta r_2}+a_1e^{-\beta r_1}(1-d_2e^{-\beta r_2})\\0&(1-d_1e^{-\beta r_1})(1-d_2e^{-\beta r_2})\end{pmatrix}.
\]

For the first extreme point $\epsilon=(y_u^{-1},0)$, we get $m=(1,0)$ and the corresponding KMS$_\beta$ state $\phi_1$ satisfies 
\begin{equation}\label{values1onvertices}
\begin{pmatrix}\phi_1(t_u)\\\phi_1(t_v)\end{pmatrix}=
\begin{pmatrix}1\\
0
\end{pmatrix}.
\end{equation} 
Lemma~6.2 of \cite{AaHR} (for example) implies that $\phi$ factors through a state of the quotient by the ideal of $\TC^*(\Lambda)$ generated by $t_v$, which is the ideal denoted $I_{\{v\}}$ in \cite[\S2.4]{FaHR}. Thus the quotient is $\TC^*(\Lambda\backslash \{v\})=\TC^*(u\Lambda u)$. The general theory of \cite{aHLRS2} says that $(\TC^*(u\Lambda u),\alpha^r)$ has a unique KMS$_\beta$ state $\psi$, and we therefore have $\phi_1=\psi\circ q_{\{v\}}$, wher $q_{\{w\}}$ is the quotient map of $\TC^*(\Lambda)$ onto $\TC^*(\Lambda\backslash\{w\})$ for the hereditary subset $\{w\}$ of $\Lambda^0$ from \cite[Proposition~2.2]{aHKR2}. 

Now we consider the other extreme point $\epsilon=(0,y_v^{-1})$. This yields a KMS$_\beta$ state $\phi_2$ such that
\begin{equation}\label{valuesonvertices}
\begin{pmatrix}\phi_2(t_u)\\\phi_2(t_v)\end{pmatrix}=
\begin{pmatrix}y_v^{-1}\big(\textstyle{\prod_{i=1}^2(1-d_ie^{-\beta r_i})^{-1}}\big)\big(a_2e^{-\beta r_2}+a_1e^{-\beta r_1}-a_1d_2e^{-\beta(r_1+r_2)}\big)\\
y_v^{-1}
\end{pmatrix}.
\end{equation}
Because this vector $\epsilon$ is supported on the absolute source $v$, Proposition~8.2 of \cite{aHKR2} implies that $\phi_2$ factors through a state of $(C^*(\Lambda), \alpha^r)$ (and we can also verify this directly --- see the remark below).

We summarise our findings as follows.

\begin{prop}
Suppose that $\Lambda$, $r$ and $\beta$ are as described at the start of the section. Then $(\TC^*(\Lambda),\alpha^r)$ has a 1-dimensional simplex of KMS$_\beta$ states with extreme points $\phi_1$ and $\phi_2$ satisfying \eqref{values1onvertices} and \eqref{valuesonvertices}. The KMS state $\phi_1$ factors through a state $\psi$ of $\TC^*(u\Lambda u)$, and the KMS state $\phi_2$ factors through a state of $C^*(\Lambda)$.
\end{prop}

\begin{rmk}\label{reality2verts} 
At this stage we can do some reassuring reality checks. First, we check that $\phi_2(t_u)+\phi_2(t_v)=1$. We multiply through by $y_v$ to take the $y_v^{-1}$ out. Then we compute using that $a_1d_2=a_2d_1$:
\begin{align*}
y_v\phi_2(t_u)&+y_v\phi_2(t_v)=y_v\phi_2(t_u)+1\\
&=\big(\textstyle{\prod_{i=1}^2(1-d_ie^{-\beta r_i})^{-1}}\big)\big(a_2e^{-\beta r_2}+a_1e^{-\beta r_1}-a_1d_2e^{-\beta(r_1+r_2)}\big)+1\\
&=\big(\textstyle{\prod_{i=1}^2(1-d_ie^{-\beta r_i})^{-1}}\big)\big(a_2e^{-\beta r_2}+a_1e^{-\beta r_1}-a_2e^{\beta r_2}d_1e^{-\beta r_1}\big)+1\\
&=\big(\textstyle{\prod_{i=1}^2(1-d_ie^{-\beta r_i})^{-1}}\big)\big(a_2e^{-\beta r_2}(1-d_1e^{-\beta r_1})+a_1e^{-\beta r_1}\big)+1\\
&=a_2e^{-\beta r_2}(1-d_2e^{-\beta r_2})^{-1}+a_1e^{-\beta r_1}(1-d_1e^{-\beta r_1})^{-1}(1-d_2e^{-\beta r_2})^{-1}+1,
\end{align*}
which is the formula for $y_v$ reshuffled.

Next, we verify directly that $\phi_2$ factors through a state of $C^*(\Lambda)$. We saw in Example~\ref{ex2graph1source} that the only finite exhaustive subset of $u\Lambda^1$ is $u\Lambda^1$, and then Corollary~\ref{FEex2.2} implies that 
\begin{equation}\label{expandKMS2.2}
\phi_{2}\Big(\prod_{e\in u\Lambda^1}(t_u-t_et_e^*)\Big)=
\phi_2\Big(t_u-\sum_{e\in u\Lambda^1}t_et_e^*-\sum_{f\in u\Lambda^1}t_ft_f^*+\sum_{\mu\in u\Lambda^{e_1+e_2}}t_\mu t_\mu^*\Big).
\end{equation}
Now we break each sum into two sums over subsets of $u\Lambda u$ and $u\Lambda v$, and apply the KMS condition to each $\phi_{2}(t_\mu t_\mu^*)=e^{-\beta r\cdot d(\mu)}\phi_2(t_{s(\mu)})$. We find that \eqref{expandKMS2.2} is
\[
(1-d_1e^{-\beta r_1})(1-d_2e^{-\beta r_2})\phi_2(t_u)-\big(a_1e^{-\beta r_1}+a_2e^{-\beta r_2}-a_1d_2e^{-\beta(r_1+r_2)}\big)\phi_{2}(t_v),
\]
which vanishes by \eqref{valuesonvertices}. Now the standard argument (using, for example, \cite[Lemma~6.2]{AaHR}) shows that $\phi_2$ factors though a state of the Cuntz--Krieger algebra $C^*(\Lambda)$, which by Example~\ref{ex2graph1source} is the quotient of $\TC^*(\Lambda)$ by the single Cuntz--Krieger relation $\textstyle{\prod_{e\in u\Lambda^1}(t_u-t_et_e^*)}=0$.
\end{rmk}

\section{KMS states for the graphs of Example~\ref{trickyex}}\label{sectrickyex}

We now consider a $2$-graph $\Lambda$ with the skeleton described in Example~\ref{trickyex}. Such graphs have one absolute source $w$, and $\Lambda\backslash\{w\}$ is the graph discussed in the previous section. As usual, we consider a dynamics determined by $r\in (0,\infty)^2$, and we want to use Theorem~6.1 of \cite{aHLRS2} to find the KMS$_\beta$  states for $\beta$ satisfying $\beta r_i>\ln\rho(A_i)$.  Our first task is to find the vector $y=(y_u,y_v,y_w)$.

Since the sets $\Lambda u$ and $\Lambda v$ lie entirely in the subgraph with vertices $\{u,v\}$, the numbers $y_u:=\sum_{\mu\in \Lambda u} e^{-\beta r_i\cdot d(\mu)}$ and $y_v$ are given by Lemma~\ref{lemcompy}. So it remains to compute $y_w$. We find:

\begin{lem}\label{lemcompy2}
We define $\Delta:=(1-d_1e^{-\beta r_1})(1-d_2e^{-\beta r_2})$. Then we have
\begin{align}
y_u&=\Delta^{-1},\notag\\
y_v&=1+a_1e^{-\beta r_1}\Delta^{-1}+a_2e^{-\beta r_2}(1-d_2e^{-\beta r_2})^{-1}\notag\\
&=1+a_1e^{-\beta r_1}\Delta^{-1}+a_2e^{-\beta r_2}(1-d_1e^{-\beta r_1})\Delta^{-1},\quad \text{and}\label{ysubv2}\\
y_w&=1+b_2e^{-\beta r_2}+b_1e^{-\beta r_1}\Delta^{-1}+a_2b_2e^{-2\beta r_2}(1-d_2e^{-\beta r_2})^{-1}\label{ysubw}.
\end{align}
\end{lem} 

\begin{proof}
As foreshadowed above, the formula for $y_u$ and the first formula for $y_v$ follow from Lemma~\ref{lemcompy}. The formula \eqref{ysubv2} is just a rewriting of the previous one which will be handy in computations (and this trick will be used a lot later). 

To find $y_w$, we consider the paths $\mu=\nu e$ with $e\in \Lambda^{e_1}w$ and $\nu\in \Lambda u$ (these are the ones with $d(\mu)\geq e_1$). There are $b_1$ such $e$, and hence we have a contribution $b_1e^{-\beta r_1}y_u=b_1e^{-\beta r_1}\Delta^{-1}$ to $y_w$. The remaining paths are in $\Lambda^{\NN e_2}w$, and give a contribution of 
\begin{align*}
1+b_2e^{-\beta r_2}+b_2e^{-\beta r_2}&a_2e^{-\beta r_2}\sum_{l=0}^\infty d_2^le^{-\beta r_2l}\\&=1+b_2e^{-\beta r_2}+b_2e^{-\beta r_2}a_2e^{-\beta r_2}(1-d_2e^{-\beta r_2})^{-1}.
\end{align*}
Adding the two contributions gives \eqref{ysubw}.
\end{proof}

\begin{rmk}
We could also have computed $y_w$ by counting the paths with $d(\mu)\geq e_2$ and those in $\Lambda^{\NN e_1}w$. This gives 
\begin{equation}\label{ywalt}
y_w=1+b_1e^{-\beta r_1}(1-d_1e^{-\beta r_1})^{-1}+b_2e^{-\beta r_2}y_v.
\end{equation}
We found the check that this is the same as the right-hand side of \eqref{ysubw} instructive. First, we use the alternative formula \eqref{altformyv} for $y_v$ (whose proof in Remark~\ref{rmkaltformyv} used the crucial identity $a_1d_2=a_2d_1$). Then 
the right-hand side of \eqref{ywalt} becomes
\begin{align*}
1+b_1e^{-\beta r_1}&(1-d_1e^{-\beta r_1})^{-1}\\
&+b_2e^{-\beta r_2}\big(1+a_2e^{-\beta r_2}\Delta^{-1}+a_1e^{-\beta r_1}(1-d_1e^{-\beta r_1})^{-1}\big).
\end{align*}
Now we write $(1-d_1e^{-\beta r_1})^{-1}=(1-d_2e^{-\beta r_2})\Delta^{-1}$, similarly for $(1-d_1e^{-\beta r_1})^{-1}$, and expand the brackets: we get
\begin{align*}
1+b_1&e^{-\beta r_1}\Delta^{-1}-b_1d_2e^{-\beta(r_1+r_2)}\Delta^{-1}\\
&+b_2e^{-\beta r_2}+b_2a_2e^{-2\beta r_2}\Delta^{-1}+b_2a_1e^{-\beta(r_1+r_2)}\Delta^{-1}-b_2a_1d_2e^{-\beta(r_1+2r_2)}\Delta^{-1}.
\end{align*}
Now we recall from Example~\ref{trickyex} that $b_1d_2=b_2a_1$, and hence the third and sixth terms cancel. Next we use the identity $a_1d_2=a_2d_1$ in the last term. We arrive at
\begin{align*}
1+b_1e^{-\beta r_1}\Delta^{-1}&+b_2e^{-\beta r_2}+b_2a_2e^{-2\beta r_2}\Delta^{-1}
-b_2a_2d_1e^{-\beta(r_1+2r_2)}\Delta^{-1}\\
&=1+b_1e^{-\beta r_1}\Delta^{-1}+b_2e^{-\beta r_2}+b_2a_2e^{-2\beta r_2}(1-d_1e^{-\beta r_1})\Delta^{-1}\\
&=1+b_1e^{-\beta r_1}\Delta^{-1}+b_2e^{-\beta r_2}+b_2a_2e^{-2\beta r_2}(1-d_2e^{-\beta r_2})^{-1},
\end{align*}
which is the the formula for $y_w$ in \eqref{ysubw}. We find it reassuring that we had to explicitly use both relations $b_1d_2=b_2a_1$ and $a_1d_2=a_2d_1$ that are imposed on us by the assumption that our coloured graph is the skeleton of a $2$-graph.
\end{rmk}

Theorem~6.1 of \cite{aHLRS2} says that for each $\beta$ satisfying $\beta r_i>\ln\rho(A_i)$ for $i=1,2$, there is a simplex of KMS$_\beta$ states $\phi_\epsilon$ on $(\TC^*(\Lambda),\alpha^r)$ parametrised by the set
\[
\Delta_\beta=\big\{\epsilon\in [0,\infty)^{\{u,v,w\}}:\epsilon\cdot y=1\big\}.
\]
Here, the set $\Delta_\beta$ is a $2$-dimensional simplex with extreme points $e_u:=(y_u^{-1},0,0)$, $e_v:=(0,y_v^{-1},0)$, and $e_w=(0,0,y_w^{-1})$. The values of $\phi_\epsilon$  on the vertex projections are the entries in the vector $m(\epsilon)=\prod_{i=1}^2(1-e^{-\beta r_i}A_i)^{-1}\epsilon$. Since the matrices $1-e^{-\beta r_i}A_i$ are upper-triangular, so are their inverses, and we deduce that both $m(e_u)$ and $m(e_v)$ have final entry $m(e_u)_w=0=m(e_v)_w$. So the corresponding KMS states are the compositions of the states of $\big(\TC^*(\Lambda\backslash\{w\}), \alpha^{(r_1,r_2)}\big)$ with the quotient map $q_{\{w\}}:\TC^*(\Lambda)\to \TC^*(\Lambda\backslash\{w\})$.  Thus the extreme points of the simplex of KMS$_\beta$ states of $(\TC^*(\Lambda),\alpha^r)$ are $\phi_1\circ q_{\{w\}}=(\psi\circ q_{\{v\}})\circ q_{\{w\}}$, $\phi_2\circ q_{\{w\}}$ and $\psi_3:=\phi_{e_w}$.  

\begin{rmk} 
We recall from the end of the previous section that the state $\phi_2$ of $\TC^*(\Lambda\backslash\{w\})$ factors through a state of the Cuntz--Krieger algebra $C^*(\Lambda\backslash\{w\})$. So it is tempting to ask whether $\phi_2\circ q_{\{w\}}$ factors through a state of $C^*(\Lambda)$. Hoewever, this is  not the case. The point is that in the graph $\Lambda\backslash\{w\}$, the vertex $v$ is an absolute source, and hence there is no Cuntz--Krieger relation involving $t_v$. However, in the larger graph $\Lambda$, $v$ is not an absolute source: the set $v\Lambda^{e_2}$ is a nontrivial finite exhaustive subset of $v\Lambda^1$, and hence the Cuntz--Krieger family generating $C^*(\Lambda)$ must satisfy the relation
\[
\prod_{e\in v\Lambda^2}(t_v-t_et_e^*)=0\Longleftrightarrow t_v-\sum_{e\in e\Lambda^{e_2}}t_et_e^*=0.
\]
The KMS condition implies that the state $\phi:=\phi_2\circ q_{\{w\}}$ satisfies 
\[
\phi(t_et_e^*)=e^{-\beta r_2}\phi(t_{s(e)})=e^{-\beta r_2}\phi(t_w)=e^{-\beta r_2}\phi_2\circ q_{\{w\}}(t_w)=e^{-\beta r_2}\phi_2(0)=0
\]
for all $e\in v\Lambda^{e_2}$. Since we know from \eqref{valuesonvertices} that $\phi(t_v)=\phi_2(t_v)=y_v^{-1}$ is not zero, we deduce that 
\[
\phi\Big(t_v-\sum_{e\in v\Lambda^{e_2}}t_et_e^*\Big)=\phi(t_v)\not=0.
\] 
Thus $\phi=\phi_2\circ q_{\{w\}}$ does not factor through a state of $C^*(\Lambda)$.
\end{rmk} 

We now focus on the new extreme point is $\phi_{e_w}$. To compute it, we need to calculate $\prod_{i=1}^2(1-e^{-\beta r_i}A_i)^{-1}$. Since the matrices $A_1$ and $A_2$ commute, so do the matrices $1-e^{-\beta r_i}A_i$, and it suffices to compute the inverse of 
\begin{align*}
\prod_{i=1}^2(1&-e^{-\beta r_i}A_i)=\\
&=\begin{pmatrix}
\Delta&-(1-d_1e^{-\beta r_1})a_2e^{-\beta r_2}-a_1e^{-\beta r_1}&a_1b_1b_2^{-\beta(2r_1+r_2)}-b_1e^{-\beta r_1}1\\
0&1&-b_2e^{-\beta r_2}\\
0&0&1
\end{pmatrix},
\end{align*}
where as before we write $\Delta=\prod_{i=1}^2(1-d_ie^{-\beta r_i})$. We find that the inverse is 
\[
\Delta^{-1}\begin{pmatrix}1&(1-d_1e^{-\beta r_1})a_2e^{-\beta r_2}+a_1e^{-\beta r_1}&(1-d_1e^{-\beta r_1})a_2b_2e^{-2\beta r_2}+b_1e^{-\beta r_1}\\
0&\Delta&\Delta b_2e^{-\beta r_2}\\
0&0&\Delta
\end{pmatrix}.
\]
Thus the corresponding KMS$_\beta$ state $\phi_{e_w}$ satisfies
\begin{equation}\label{charphi3}
\begin{pmatrix}
\phi_{e_w}(t_u)\\ \phi_{e_w}(t_v)\\ \phi_{e_w}(t_w)
\end{pmatrix}=\begin{pmatrix}
\Delta^{-1}\big((1-d_1e^{-\beta r_1})a_2b_2e^{-2\beta r_2}+b_1e^{-\beta r_1}\big) y_w^{-1}\\
b_2e^{-\beta r_2}y_w^{-1}\\
y_w^{-1}
\end{pmatrix}.
\end{equation}

\begin{rmk}
As usual, we take the opportunity for a reality check: since $t_u+t_v+t_w$ is the identity of $\TC^*(\Lambda)$ and $\phi_{e_w}$ is a state, we must have $\phi_{e_w}(t_u)+\phi_{e_w}(t_v)+\phi_{e_w}(t_w)=1$. But since $\Delta^{-1}(1-d_1e^{-\beta r_1})=(1-e^{-\beta r_2})^{-1}$, the formula \eqref{ysubw} says that this sum is precisely $y_wy_w^{-1}=1$.
\end{rmk}

We summarise our findings in the following theorem.

\begin{thm}\label{KMSabovecrit}
Suppose that $\Lambda$ is a $2$-graph with skeleton described in Example~\ref{trickyex} and vertex matrices $A_1$,  $A_2$. We suppose that $r\in (0,\infty)^{\{u,v,w\}}$, and consider the dynamics $\alpha^r$ on $\TC^*(\Lambda)$.  We suppose that  $\beta>0$ satisfies $\beta r_i>\ln\rho(A_i)$ for $i=1,2$. We write $\phi_1$ and $\phi_2$ for the KMS$_\beta$ states of $(\TC^*(\Lambda\backslash\{w\}), \alpha^{r})$ described before Remark~\ref{reality2verts}. Then $\phi_1\circ q_{\{w\}}$ and $\phi_2\circ q_{\{w\}}$ are KMS$_\beta$ states of $(\TC^*(\Lambda),\alpha^r)$. There is another KMS$_\beta$ state $\phi_3=\phi_{e_w}$ satisfying \eqref{charphi3}. Every KMS$_\beta$ state of $(\TC^*(\Lambda),\alpha^r)$ is a convex combination of the three states $\phi_1\circ q_{\{w\}}$, $\phi_2\circ q_{\{w\}}$ and $\phi_3$. None of these KMS$_\beta$ states factors through a state of $(C^*(\Lambda),\alpha^r)$. 
\end{thm}

The only thing we haven't proved is the assertion that every KMS state is a convex combination of the states that we have described. But this follows from the general results in \cite[Theorem~6.1]{aHLRS2}, because the vectors $(y_u^{-1},0,0)$, $(0,y_v^{-1},0)$ and $(0,0,y_w^{-1})$ are the extreme points of the simplex $\Delta_\beta$. 
  
\section{KMS states at the critical inverse temperature}\label{seccritinvt}

We begin with the graphs of Example~\ref{ex2graph1source}. We observe that the hypothesis of rational independence in the two main results of this section is in practice easy to verify using Proposition~A.1 of \cite{aHKR2}: loosely, $\ln d_1$ and $\ln d_2$ are rationally independent unless $d_1$ and $d_2$ are different powers of the same integer. 

\begin{prop}\label{KMS1on2.2}
Suppose that $\Lambda$ is a 2-graph with the skeleton described in Example~\ref{ex2graph1source} and that $r\in (0,\infty)^2$ has $r_i\geq\ln  d_i$ for both $i$,  $r_i=\ln d_i$ for at least one $i$, and $\{r_1,r_2\}$ are rationally independent. Consider the quotient map $q_{\{v\}}:\TC^*(\Lambda)\to \TC^*(\Lambda \backslash\{v\})$ from \cite[Proposition~2.2]{aHKR2}. Then $(\TC^*(\Lambda \backslash\{v\}),\alpha^r)$ has a unique KMS$_1$ state $\phi$, and $\phi\circ q_{\{v\}}$ is the only KMS$_1$ state of $(\TC^*(\Lambda),\alpha^r)$. 
\end{prop}

\begin{lem}\label{KMS1on2.2lem}
Suppose that $\Lambda$  and $r$ are as in Proposition~\ref{KMS1on2.2}, and that $\phi$ is a KMS$_1$ state $\phi$ of $(\TC^*(\Lambda),\alpha^r)$. Then $\phi(t_v)=0$.  
\end{lem}

\begin{proof}
We recall from Example~~\ref{ex2graph1source} that the only finite exhaustive subset of $u\Lambda^1$ is $u\Lambda^1$ itself, and from Example~\ref{ex2graph1sourceFE} we then have
\[
\prod_{e\in u\Lambda^1}(t_u-t_et_e^*)=t_u-\sum_{e\in u\Lambda^1}t_et_e^*+\sum_{\mu\in u\Lambda^{e_1+e_2}} t_\mu t_\mu^*. 
\]
Thus positivity of $\phi\big(\prod_{e\in u\Lambda^1}(t_u-t_et_e^*)\big)$ implies that
\[
0\leq \phi(t_u)-\sum_{e\in u\Lambda^{e_1}}\phi(t_et_e^*)-\sum_{e\in u\Lambda^{e_2}}\phi(t_et_e^*)+\sum_{\mu\in u\Lambda^{e_1+e_2}}\phi(t_\mu t_\mu^*).
\]
Now we use the KMS relation and count paths of various degrees to get
\begin{align*}
0&\leq \phi(t_u)-\sum_{e\in u\Lambda^{e_1}}e^{-r_1}\phi(t_{s(e)})-\sum_{e\in u\Lambda^{e_2}}e^{-r_2}\phi(t_{s(e)})+\sum_{\mu\in u\Lambda^{e_1+e_2}}e^{-(r_1+r_2)}\phi(t_{s(\mu)})\\
&=\phi(t_u)-e^{-r_1}\big(d_1\phi(t_u)+a_1\phi(t_v)\big)-e^{-r_2}\big(d_2\phi(t_u)+a_2\phi(t_v)\big)\\
&\hspace{7cm}+e^{-(r_1+r_2)}\big(d_1d_2\phi(t_u)+d_1a_2\phi(t_v)\big)\\
&=\big(1-d_1e^{-r_1}-d_2e^{-r_2}+d_1d_2e^{-(r_1+r_2)}\big)\phi(t_u)\\
&\hspace{5cm}-\big(a_1e^{-r_1}+a_2e^{-r_2}-d_1a_2e^{-(r_1+r_2)}\big)\phi(t_v)\\
&=(1-d_1e^{-r_1})(1-d_2e^{-r_2})\phi(t_u)-\big(a_1e^{-r_1}+a_2e^{-r_2}-d_1a_2e^{-(r_1+r_2)}\big)\phi(t_v)\\
&=-\big(a_1e^{-r_1}+a_2e^{-r_2}-d_1a_2e^{-(r_1+r_2)}\big)\phi(t_v),
\end{align*}
where the coefficient of $\phi(t_u)$ vanished because for at least one of $i=1,2$ we have  $1-d_ie^{-r_i}=1-d_id_i^{-1}=0$ by the hypotheses on $r_i$. If $r_1=\ln d_1$, then we write this as
\[
0\leq-\big(a_1e^{-r_1} +(1-d_1e^{-r_1})a_2e^{-r_2}\big)\phi(t_v)=-a_1e^{-r_1}\phi(t_v),
\]
and positivity of $\phi(t_v)$ implies that $t_v=0$. If $r_2=\ln d_2$, then we use the identity $d_1a_2=d_2a_1$ to rewrite it as 
\[
0\leq -\big(a_2e^{-r_2} +(1-d_2e^{-r_2})a_1e^{-r_1}\big)\phi(t_v)=-a_2e^{-r_2}\phi(t_v),
\]
which also implies that $\phi(t_v)=0$.
\end{proof} 

\begin{proof}[Proof of Proposition~\ref{KMS1on2.2}]
Proposition~4.2 of \cite{aHKR} implies that $(\TC^*(\Lambda \backslash\{v\},\alpha^r)$ has a unique KMS$_1$ state $\phi$. Since $q_{\{v\}}$ interwines the two actions $\alpha^r$, the composition $\phi\circ q_{\{v\}}$ is a KMS$_1$ state of $(\TC^*(\Lambda),\alpha^r)$. On the other hand, if $\psi$ is a KMS$_1$ state of $(\TC^*(\Lambda),\alpha^r)$, then Lemma~\ref{KMS1on2.2lem} implies that $\psi(t_v)=0$. The standard argument using \cite[Lemma~6.2]{AaHR} shows that $\psi$ factors through the quotient by the ideal generated by $t_v$, which is precisely the kernel of $q_{\{v\}}$. Thus there is a KMS$_1$ state $\theta$ of $(\TC^*(\Lambda \backslash\{v\},\alpha^r)$ such that $\psi=\theta\circ q_{\{v\}}$, and uniqueness of the KMS$_1$ state implies that $\theta=\phi$. Hence $\psi= \phi\circ q_{\{v\}}$, as required.
\end{proof}

\begin{thm}\label{KMS1on2.9}
Suppose that $\Lambda$ is a 2-graph with the skeleton described in Example~\ref{ex2graph1sourceFE} and that $r\in (0,\infty)^2$ has $r_i\geq\ln  d_i$ for both $i$,  $r_i=\ln d_i$ for at least one $i$, and $\{r_1,r_2\}$ are rationally independent. Consider the quotient map $q_{\{v,w\}}$ of $\TC^*(\Lambda)$ onto $\TC^*(\Lambda \backslash\{v,w\})$ discussed in \cite[Proposition~2.2]{aHKR2}. Then $(\TC^*(\Lambda \backslash\{v,w\}),\alpha^r)$ has a unique KMS$_1$ state $\phi$, and $\phi\circ q_{\{v,w\}}$ is the only KMS$_1$ state of $(\TC^*(\Lambda),\alpha^r)$. 
\end{thm}

We need an analogue of Lemma~\ref{KMS1on2.2lem} for the present situation. 

\begin{lem}\label{KMS1on2.9lem}
Under the hypotheses of the preceding theorem, suppose that $\phi$ is a KMS$_1$ state of $(\TC^*(\Lambda),\alpha^r)$. Then $\phi(t_w)=0$. 
\end{lem}

\begin{proof}
We use an argument like that in the proof of Lemma~\ref{KMS1on2.2lem} for the exhaustive subset $E=(u\Lambda^1u)\cup(u\Lambda^{e_2}v)\cup(u\Lambda^{e_1}w)$ of $u\Lambda^1$ discussed in Example~\ref{trickyex}. Then the state satisfies
\[
\phi\big(\textstyle{\prod_{e\in E}}(t_u-t_et_e^*)\big)\geq 0.
\]
Now using Lemma~\ref{CKatu} to write $\textstyle{\prod_{e\in E}}(t_u-t_et_e^*)$ as a sum gives
\[
\phi\Big(t_u-\sum_{e\in u\Lambda^{e_1}\{u,w\}}t_et_e^*-\sum_{f\in u\Lambda^{e_2}\{u,v\}}t_ft_f^*+\sum_{\mu\in u\Lambda^{e_1+e_2}\{u,v\}}t_\mu t_\mu^*\Big)\geq 0.
\]
Now we use linearity of $\phi$ and the KMS condition to get
\begin{align*}
0&\leq \phi(t_u)-\sum_{e\in u\Lambda^{e_1}\{u,w\}}\phi(t_et_e^*)-\sum_{f\in u\Lambda^{e_2}\{u,v\}}\phi(t_ft_f^*)+\sum_{\mu\in u\Lambda^{e_1+e_2}\{u,v\}}\phi(t_\mu t_\mu^*)\\
&=\phi(t_u)-\big(d_1e^{-r_1}\phi(t_u)+b_1e^{-r_1}\phi(t_w)\big)-\big(d_2e^{-r_2}\phi(t_u)+a_2e^{-r_2}\phi(t_v)\big)\\
&\hspace{5cm}+\big(d_1d_2e^{-(r_1+r_2)}\phi(t_u)+d_1a_2e^{-(r_1+r_2)}\phi(t_v)\big)\\
&=(1-d_1e^{-r_1})(1-d_2e^{-r_2})\phi(t_u)-(1-d_1e^{-r_1})a_2e^{-r_2}\phi(t_v)-b_1e^{-r_1}\phi(t_w).
\end{align*}
Since at least one $r_i$ is $\ln d_i$, we have $(1-d_1e^{-r_1})(1-d_2e^{-r_2})=0$, and we deduce that 
\[
0\leq -(1-d_1e^{-r_1})a_2e^{-r_2}\phi(t_v)-b_1e^{-r_1}\phi(t_w). 
\]
Since $(1-d_1e^{-r_1})a_2e^{-r_2}$, $\phi(t_v)$, and $b_1e^{-r_1}$ are all nonnegative, we must have both $(1-d_1e^{-r_1})a_2e^{-r_2}\phi(t_v)=0$ and $b_1e^{-r_1}\phi(t_w)=0$. In particular, we deduce that $\phi(t_w)=0$.
\end{proof}

\begin{proof}[Proof of Theorem~\ref{KMS1on2.9}]
We suppose that $\psi$ is a KMS$_1$ state of $(\TC^*(\Lambda),\alpha^r)$. Then Lemma~\ref{KMS1on2.9lem}
implies that $\phi(t_w)=0$. The formula in \cite[Proposition~2.1(1)]{aHKR2} implies that $\phi$ vanishes on the ideal $I_{\{w\}}$ generated by $t_w$, and by \cite[Lemma~6.2]{AaHR} $\phi$ factors through a KMS$_1$ state $\theta$ of the system $(\TC^*(\Lambda\backslash\{w\}),\alpha^r)$. The $2$-graph $\Lambda\backslash \{w\}$ is the graph in Proposition~\ref{KMS1on2.2}, and hence that Proposition implies that $\theta =\phi\circ q_{\{v\}}$. The kernel of the composition $q_{\{v\}}\circ q_{\{w\}}$ is the ideal generated by $\{t_v,t_w\}$, and a glance at the definition of the homomorphism in \cite[Proposition~2.2(2)]{aHKR2} shows that   $q_{\{v\}}\circ q_{\{w\}}=q_{\{v,w\}}$. Thus 
\[
\psi=\theta\circ q_{\{w\}}=(\phi\circ q_{\{v\}})\circ q_{\{w\}}=\phi\circ q_{\{v,w\}}.\qedhere
\]
\end{proof}

\section{Where our examples came from}\label{secbigex}

We consider a $2$-graph $\Lambda$ with skeleton
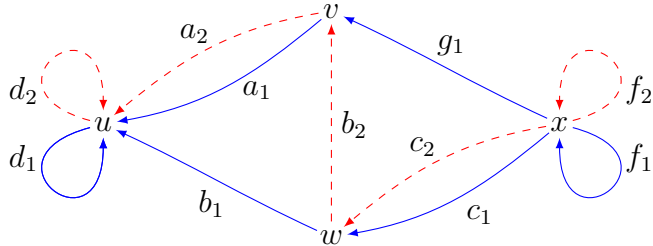
\begin{equation*}\label{4vertsex}
\begin{tikzpicture}[scale=1.5]
 \node[inner sep=0.5pt, circle] (u) at (0,0) {$u$};
    \node[inner sep=0.5pt, circle] (v) at (2,1) {$v$};
    \node[inner sep=0.5pt, circle] (w) at (2,-1) {$w$};
    \node[inner sep=0.5pt, circle] (x) at (4,0) {$x$};
\draw[-latex, blue] (u) edge [out=195, in=270, loop, min distance=30, looseness=2.5] (u);
\draw[-latex, red, dashed] (u) edge [out=165, in=90, loop, min distance=30, looseness=2.5] (u);
\draw[-latex, blue] (v) edge [out=220, in=10] (u);
\draw[-latex, red, dashed] (v) edge [out=190, in=40] (u) ;
\draw[-latex, blue] (w) edge [out=155, in=335] (u) ;
\draw[-latex, red, dashed] (w) edge [out=90, in=270] (v);
\draw[-latex, blue] (u) edge [out=195, in=270, loop, min distance=30, looseness=2.5] (u);
\draw[-latex, blue] (x) edge [out=155, in=335] (v);
\draw[-latex, blue] (x) edge [out=220, in=10] (w);
\draw[-latex, red, dashed] (x) edge [out=190, in=40] (w) ;
\draw[-latex, blue] (x) edge [out=345, in=270, loop, min distance=30, looseness=2.5] (x);
\draw[-latex, red, dashed] (x) edge [out=15, in=90, loop, min distance=30, looseness=2.5] (x);
\node at (-.7, 0.3) {\color{black} $d_2$};
\node at (-.7,-0.3) {\color{black} $d_1$};
\node at (.8, 0.8) {\color{black} $a_2$};
\node at (1.35,0.3) {\color{black} $a_1$};
\node at (2.2, 0) {\color{black} $b_2$};
\node at (.95,-0.7) {\color{black} $b_1$};
\node at (3.3,-0.8) {\color{black} $c_1$};
\node at (2.8,-0.2) {\color{black} $c_2$};
\node at (3.05,0.7) {\color{black} $g_1$};
\node at (4.7, 0.3) {\color{black} $f_2$};
\node at (4.7,-0.3) {\color{black} $f_1$};
\end{tikzpicture}
\end{equation*}
In these graphs there are two nontrivial strongly connected components $\{u\}$ and $\{x\}$, and the bridges $\mu\in u\Lambda x$ all have $|d(\mu)|>1$. The graphs in Examples~\ref{ex2graph1source} and \ref{trickyex} are then the graphs $\Lambda\backslash\{w,x\}$ and $\Lambda\backslash\{x\}$, respectively. We assume that all the integers $d_i,a_i,b_i,c_i,g_i,f_i,$ are nonzero. 

The vertex matrices of the graph $\Lambda$ are then
\begin{equation}
A_1=\begin{pmatrix}d_1&a_1&b_1&0\\0&0&0&g_1\\0&0&0&c_1\\0&0&0&f_1\end{pmatrix}\qquad\text{and}\qquad 
A_2=\begin{pmatrix}d_2&a_2&0&0\\0&0&b_2&0\\0&0&0&c_2\\0&0&0&f_2\end{pmatrix}.
\end{equation}
Thus we have $\rho(A_i)=\max\{d_i,f_i\}$ for $i=1,2$.  As in the last section, we consider a dynamics $\alpha^r:\RR\to \Aut \TC^*(\Lambda)$ given by $r\in (0,\infty)^2$ such that $r_i\geq \ln\rho(A_i)$ for $i=1,2$, and $r_i=\ln\rho(A_i)$ for at least one $i$. We are interested in the KMS$_1$ states of $(\TC^*(\Lambda),\alpha^r)$.

If $\rho(A_i)=d_i$ for some $i$, then the strongly connected component $\{u\}$ is $i$-critical in the sense of \cite[\S3]{FaHR}, and Proposition~3.1 and Corollary~3.2 in \cite{FaHR} imply that all the KMS$_1$ states factor through states of $\Lambda\backslash\{v,w,x\}=u\Lambda u$. Proposition~4.2 of \cite{aHKR2} then implies that $(\TC^*(\Lambda),\alpha^r)$ has a unique KMS$_1$ state.

So we suppose from now on that $\rho(A_i)=f_i>d_i$ for $i=1,2$. We now want to run through the construction of \cite[\S4--5]{FaHR}. The set $H$ in \cite[Proposition~4.1]{FaHR} is empty, so the block decompositions of the matrices $A_i$ look like
\[
A_i=\begin{pmatrix}E_i&B_i\\0&f_i\end{pmatrix}
\]
where $E_i$ is $3\times 3$ and $B_i$ is $3\times 1$. We can choose to work with either $i=1$ or $i=2$, and $i=2$ is marginally simpler because $B_2=\begin{pmatrix}0&0&c_2\end{pmatrix}^T$. The unimodular Perron--Frobenius eigenvector of the matrix $(f_2)$ is the number $1$, and hence the vector $y$ in \cite[Proposition~4.1]{FaHR} is
\begin{align}
y&=\big(\rho(A_{\{u\},2})1-E_2\big)^{-1}B_2\notag\\
&=\begin{pmatrix}f_2-d_2&-a_2&0\\0&f_2&-b_2\\0&0&f_2\end{pmatrix}^{-1}\begin{pmatrix}0\\0\\c_2\end{pmatrix}\notag\\
&=\frac{1}{(f_2-d_2)f_2^2}\begin{pmatrix}
f_2^2&a_2f_2&a_2b_2\\0&(f_2-d_2)f_2&(f_2-d_2)b_2\\0&0&(f_2-d_2)f_2\end{pmatrix}\begin{pmatrix}0\\0\\c_2\end{pmatrix}\notag\\
&=\frac{1}{(f_2-d_2)f_2^2}\begin{pmatrix}
a_2b_2c_2\\(f_2-d_2)b_2c_2\\
(f_2-d_2)f_2c_2\end{pmatrix}=\begin{pmatrix}
(f_2-d_2)^{-1}f_2^{-2}a_2b_2c_2\\f_2^{-2}b_2c_2\\
f_2^{-1}c_2\end{pmatrix}\label{expB_2}.
\end{align}

\begin{rmk}
As we said above, we should have been able to work with $i=1$ and get the same answer (see Equation (4.3) in \cite[Proposition~4.1]{FaHR}). This gives us another opportunity for a reality check. But the answer we got the second time looked quite different, and in sorting out the mess we learned something interesting. The second answer, in the form we first got it, was
\begin{equation}\label{expB_1}
y=\frac{1}{(f_1-d_1)f_1^2}\begin{pmatrix}
a_1f_1g_1+f_1b_1c_1\\(f_1-d_1)f_1g_1\\
(f_1-d_1)f_1g_1\end{pmatrix}=\begin{pmatrix}
(f_1-d_1)^{-1}f_1^{-1}(a_1g_1+b_1c_1)\\f_1^{-1}g_1\\
f_1^{-1}c_1\end{pmatrix}.
\end{equation}
Equality of the third entries in \eqref{expB_2} and \eqref{expB_1} is equivalent to  $f_1c_2=f_2c_1$, which is one of the relations imposed by the requirement that $\Lambda$ is a $2$-graph. Similar reasoning works for the second entries. But when we removed the inverses by cross-multiplying, the top entry in the second calculation became
\begin{equation}\label{badform}
(f_2-d_2)f_2^2(a_1g_1+b_1c_1)=f_2^3a_1g_1+f_2^3b_1c_1-d_2f_2^2a_1g_1-d_2f_2^2b_1c_1.
\end{equation}
In the other calculation, the top entry has just two summands, which are again products of 5 terms. After staring at them for a bit, we realised that these products have meaning: for example, $f_2^3a_1g_1$ is the number $a_1g_1f_2^3$ of paths in $u\Lambda^{2e_1+3e_2}x$ counted using their BBRRR factorisations. Similarly, $d_2f_2^2b_1c_1=d_2b_1c_1f_2^{2}$ counts the same set using the RBBRR factorisations. Now looking at the skeleton confirms that
\[
a_1g_1f_2^3=a_1(g_1f_2)f_2^2=a_1(b_2c_1)f_2^2=(a_1b_2)c_1f_2^2=d_2b_1c_1f_2^2,
\]
and the first and last terms on the right of \eqref{badform} cancel. Similar considerations using the $1$-skeleton match up the remaining terms in the top entries in \eqref{expB_2} and \eqref{expB_1}.
\end{rmk}

We now use the results of \cite[\S5]{FaHR} to describe all the KMS$_1$ states on $(\TC^*(\Lambda),\alpha^r)$ when $\Lambda$ has skeleton described at the start of the section. First we apply \cite[Proposition~5.1]{FaHR} with $z=\big(\begin{smallmatrix}y\\1\end{smallmatrix}\big)$. This gives a KMS$_1$ state $\psi$ of $(\TC^*(\Lambda),\alpha^r)$
such that
\[
\psi(t_\mu t_\nu^*)=\delta_{\mu,\nu}e^{-r\cdot d(\mu)}\|z\|_1^{-1}z_{s(\mu)}.
\]
It factors through a KMS$_1$ state of $(C^*(\Lambda),\alpha^r)$ if and only if $r_i=\ln\rho(A_i)=f_i$ for both $i=1$ and $i=2$.

Now Theorem~5.2 of \cite{FaHR} implies that every KMS$_1$ state of $(\TC^*(\Lambda),\alpha^r)$ is a convex combination of $\psi$ and a state $\phi\circ q_{\{x\}}$ lifted from a KMS$_1$ state of $\TC^*(\Lambda\backslash \{x\})$. Since $\Lambda\backslash\{x\}$ is the graph considered in \S\ref{sectrickyex} and we are assuming that $d_i<f_i=\ln\rho(A_i)$, we are in the range where Theorem~\ref{KMSabovecrit} gives an explicit description of these KMS$_1$ states.

\section{Computing the KMS states on a specific graph}

In \cite[Example~8.4]{aHKR2}, we tested the general results in \cite{aHKR2} about the preferred dynamics by computing all the KMS states of the system $(\TC^*(\Lambda),\alpha^r)$ when $\Lambda$ has the following skeleton: 
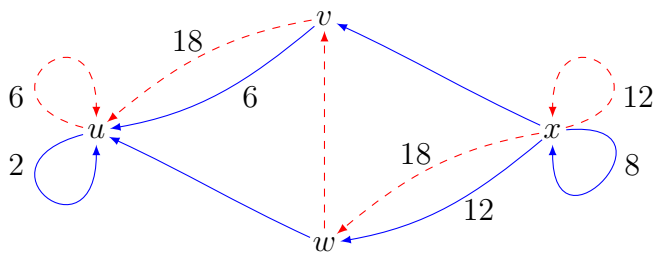
\[\begin{tikzpicture}[scale=1.5]
 \node[inner sep=0.5pt, circle] (u) at (0,0) {$u$};
    \node[inner sep=0.5pt, circle] (v) at (2,1) {$v$};
    \node[inner sep=0.5pt, circle] (w) at (2,-1) {$w$};
    \node[inner sep=0.5pt, circle] (x) at (4,0) {$x$};
\draw[-latex, blue] (u) edge [out=195, in=270, loop, min distance=30, looseness=2.5] (u);
\draw[-latex, red, dashed] (u) edge [out=165, in=90, loop, min distance=30, looseness=2.5] (u);
\draw[-latex, blue] (v) edge [out=220, in=10] (u);
\draw[-latex, red, dashed] (v) edge [out=190, in=40] (u) ;
\draw[-latex, blue] (w) edge [out=155, in=335] (u) ;
\draw[-latex, blue] (x) edge [out=155, in=335] (v);
\draw[-latex, blue] (x) edge [out=220, in=10] (w);
\draw[-latex, red, dashed] (x) edge [out=190, in=40] (w);
\draw[-latex, red, dashed] (w) edge [out=90, in=270] (v);
\draw[-latex, blue] (x) edge [out=365, in=270, loop, min distance=30, looseness=2.5] (x);
\draw[-latex, red, dashed] (x) edge [out=15, in=90, loop, min distance=30, looseness=2.5] (x);
\node at (4.75, 0.3) {\color{black} $12$};
\node at (4.7,-0.3) {\color{black} $8$};
\node at (2.8, -0.2) {\color{black} $18$};
\node at (3.35,-0.7) {\color{black} $12$};
\node at (-.7, 0.3) {\color{black} $6$};
\node at (-.7,-0.3) {\color{black} $2$};
\node at (.8, 0.8) {\color{black} $18$};
\node at (1.35,0.3) {\color{black} $6$};
\end{tikzpicture}
\]
Here we do the same for a non-preferred dynamics on the same graph algebra using the stronger versions in \cite{FaHR}. We consider a dynamics $\alpha^r$ in which  \begin{equation}\label{specifyr}
r_1=\ln 8
\quad\text{and}\quad r_2>\ln 12.
\end{equation}
For $\beta>1$, \cite[Theorem~6.1]{aHLRS2} describes a $3$-dimensional simplex of KMS$_\beta$ states of $(\TC^*(\Lambda),\alpha^r)$.

Now we consider the KMS$_1$ states, and aim to apply the results of \cite{FaHR}. The common unimodular Perron--Frobenius eigenvector of $A_{\{x\},i}$ is the $1$-vector $1$, and as in of \cite[Proposition~4.1]{FaHR}, this extends to a common eigenvector $z:=(y,1)$ of the matrices $A_i$ with eigenvalue $\rho(A_{\{x\},1})=\ln 8$ and $\rho(A_{\{x\},2})=\ln 12$. (Since \cite[Proposition~4.1]{FaHR} is linear-algebraic, it applies verbatim here.) Now Proposition~5.1 of \cite{FaHR} gives a KMS$_1$ state $\psi$ of $(\TC^*(\Lambda),\alpha^r)$ with 
\[
\begin{pmatrix}\psi(q_u)\\ \psi(q_v)\\ \psi(q_w)\\ \psi(q_x)\end{pmatrix}
=\|z\|_1^{-1}z=\frac{1}{24}\begin{pmatrix}3\\1\\12\\8\end{pmatrix}.
\]
Since the only critical component of $\Lambda$ for the dynamics $\alpha^r$ is $\{x\}$, Theorem~6.1 of~\cite{FaHR} implies that every KMS$_1$ state of $(\TC^*(\Lambda),\alpha^r)$ is a convex combination of $\psi$ and a state $\phi\circ q_{\{x\}}$ lifted from a KMS$_1$ state $\phi$ of $(\TC^*(\Lambda\backslash\{x\}),\alpha^r)$.

The graph $\Lambda\backslash\{x\}=\Lambda_{\{u,v,w\}}$ is one of those we studied in \S\ref{sectrickyex}. Since
\[
r_1=\ln 8>\rho(A_{\{u,v,w\},1})=\ln 2\quad\text{and}\quad r_2>\ln 12>\rho(A_{\{u,v,w\},2})=\ln 6,
\]
$\beta=1$ is in the range for which Theorem~\ref{KMSabovecrit} gives a concrete description of the KMS$_1$ states of $(\TC^*(\Lambda\backslash\{x\}),\alpha^r)$. So the original system has a $3$-dimensional simplex of KMS$_1$ states with extreme points $\psi$, $\phi_1\circ q_{\{w,x\}}$, $\phi_2\circ q_{\{w,x\}}$ and $\phi_3\circ q_{\{x\}}$.

With the next lemma, we can continue below the inverse temperature $\beta=1$. 
\begin{lem}
Suppose that $\beta<1$ and $\phi$ is a KMS$_\beta$ state of the system $(\TC^*(\Lambda),\alpha^r)$ considered above. Then $\phi$ factors through the quotient map $q_{\{x\}}$.
\end{lem}

\begin{proof}
We aim to prove that $\phi(t_x)=0$. We certainly have $\phi(t_x)\geq 0$. The relation (T4) with $n=e_1$ implies that 
\begin{align*}
\phi(t_x)&\geq\sum_{e\in x\Lambda^{e_1}}\phi(t_et_e^*)=\sum_{e\in x\Lambda^{e_1}}e^{-\beta r_1}\phi(t_e^*t_e)\\
&=e^{-\beta (\ln 8)}|x\Lambda^{e_1}|\phi(t_x)=8^{-\beta}.8\phi(t_x)=8^{1-\beta}\phi(t_x),
\end{align*}
which we can rewrite as $(1-8^{1-\beta})\phi(t_x)\geq 0$. But $\beta<1$ implies that $8^{1-\beta}>1$, and this is only compatible with $\phi(t_x)\geq 0$ if $\phi(t_x)=0$. 
Now it follows from \cite[Lemma~6.2]{AaHR} that $\phi$ factors through $q_{\{x\}}$. 
\end{proof}

So we are interested in the KMS$_\beta$ states of $(\TC^*(\Lambda_{\{u,v,w\}}),\alpha^r)$ for $\beta<1$. Recall from the start of the section that we are assuming that $r_1=\ln 8$ and $r_2>\ln 12$.  The next critical level is
\begin{equation}\label{2ndcrit}
\beta_c:=\max\big\{(\ln 8)^{-1}\ln2,r_2^{-1}\ln 6\big\}=\max\big\{3^{-1}, r_2^{-1}\ln 6\}.
\end{equation}

For $\beta$ satisfying $\beta_c<\beta<1$, we deduce from Theorem~\ref{KMSabovecrit} that the KMS$_\beta$ states of $(\TC^*(\Lambda_{\{u,v,w\}}),\alpha^r)$ form a $2$-dimensional simplex;  Theorem~\ref{KMSabovecrit} also provides explicit formulas for the extreme points. Composing with $q_{\{w\}}$ gives a two-dimensional simplex of KMS$_\beta$ states of $(\TC^*(\Lambda),\alpha^r)$.

\begin{rmk}
Strictly speaking, to apply Theorem~\ref{KMSabovecrit} we need to scale the dynamics to ensure that the critical inverse temperature is $1$ rather than $\beta_c$. Lemma~2.1 of \cite{aHKR} gives the formulas which achieve this. We will assume that this can be done mentally (or at least ``in principle'').
\end{rmk}

For $\beta=\beta_c$, at least one of $r_i\beta\geq \ln\rho(A_{\{u\},i})$ becomes an equality. Provided $\{r_1,r_2\}$ are rationally independent,  Theorem~\ref{KMS1on2.9} implies that $(\TC^*(\Lambda_{\{u,v,w\}}),\alpha^r)$ has a unique KMS$_{\beta_c}$ state which factors through a state of $(\TC^*(\Lambda_{\{u\}}),\alpha^r)$. It follows from \cite[Proposition~6.1]{aHKR} that this state factors through a state of $(C^*(\Lambda_{\{u\}}),\alpha^r)$ if and only if we have $r_i\beta=\ln\rho(A_{\{u\},i})$ for both $i$. For
$\beta<\beta_c$, at least one of the inequalities $r_i\beta\geq \ln\rho(A_i)$ fails, and it follows from \cite[Corollary~4.3]{aHLRS2} that there are no KMS$_\beta$ states on any of these algebras. 

\begin{rmk}
For a dynamics satisfying \eqref{specifyr}, the constraint $r_2>\ln 12$ implies that $r_2^{-1}\ln 6\leq (\ln 12)^{-1}\ln 6$, and which is the bigger in \eqref{2ndcrit} will depend on $r_2$. A calculator tells us that
\[
(\ln 12)^{-1}\ln 6=\frac{\ln 6}{\ln 2+\ln 6}\sim 0.72.
\]
Thus for small $r_2$, we have $3^{-1}<r_2^{-1}\ln 6$,  but for big $r_2$, we can have $3^{-1}>r_2^{-1}\ln 6$. So $\beta_c$ could be either value.
\end{rmk}

\end{document}